\documentclass[12pt]{amsart}
\usepackage{amsfonts}
\usepackage{mathrsfs}
\usepackage{amsmath,amssymb}
\usepackage{amsthm}
\usepackage[all]{xy}

\numberwithin{equation}{section}

\begin{normalsize}
\end{normalsize}

\def\ni{\noindent}
\def\si{\sigma}
\def\R{\mathbf{R}}
\def\Z{\mathbf{Z}}

\newtheorem{theo}{Theorem}[section]

\newtheorem{lemm}[theo]{Lemma}
\newtheorem{prop}[theo]{Proposition}
\newtheorem{rema}[theo]{Remark}
\newtheorem{defi}[theo]{Definition}

\begin{document}

\title{On the traces of elements of modular group}
\author{ Bin Wang and Xinyun Zhu}
\address{Department of Mathematics, Changshu Institute of Technology, Changshu 215500, China}
\email{binwang72@gmail.com}
\address{Department of Mathematics, University of Texas of the Permian Basin, Odessa, TX, 79762}
\email{zhu\_x@utpb.edu}
\maketitle

\begin{abstract}
{We prove a conjecture by W. Bergweiler and A. Eremenko on the traces of elements of modular group in this paper. }
\end{abstract}
\section{Introduction}
W. Bergweiler and A. Eremenko made a remarkable conjecture on the traces of elements of modular group in \cite{E}. The main result of this paper is to prove their conjecture. We expect this result to have future applications in some fields such as control theory.

Let $A =\left( \begin{array}{cc}
1 & 2  \\
0 & 1
\end{array} \right)$ and $B =\left( \begin{array}{cc}
1 & 0  \\
-2 & 1
\end{array} \right)$. These two matrices generate the free group which is called $\Gamma(2)$, the principal
congruence subgroup of level 2. With arbitrary integers $m_{j} \neq 0, n_{j} \neq 0$, consider the trace of the product
$$p_k(m_1, n_1,..., m_k, n_k) = tr(A^{m_1}B^{n_1} \cdots A^{m_k}B^{n_{k}}).$$
It is easy to see that $p_k$ is a polynomial in $2k$ variables with integer coefficients.
This polynomial can be written explicitly though the formula is somewhat complicated.

Choosing an arbitrary sequence $\sigma$ of $2k$ signs $\pm $, we make a substitution
$$ p_{k}^{\sigma}(x_1, y_1,...x_k, y_k)=p_k(\pm (1+x_1), \pm (1+y_1),\cdots , \pm (1+x_k), \pm (1+y_k)). $$
Our main theorem is the following one.
\begin{theo} \label{main}
The polynomial $p_k$, for every $k>0$, has the property that for every $\sigma$, all the coefficients of the polynomial $p_{k}^{\sigma}$ are of the same sign, that is, the sequence of coefficients of  $p_{k}^{\sigma}$ has no sign changes. \end{theo}
\ni which was conjectured by W. Bergweiler and A. Eremenko in \cite{E}.

 We prove the theorem  by induction on $k$. However it is not easy to pass from \textquotedblleft level $k$\textquotedblright  to \textquotedblleft level $k+1$\textquotedblright since that $p_k$ has the above property does not simply imply that $p_{k+1}$ has the same one. The idea here is to substitute $p_k$'s with a suitable set of polynomials containing the $p_k$'s so that the difficulty disappears. This idea is explained in section 2 (see Proposition \ref{p1}) and the theorem is showed in section 3.

 {\bf  Acknowledgment} We would like to thank Alex Eremenko for his helpful comments on the earlier draft of this paper. The first author is grateful to Jianming Chang for introducing the topic to him and for many helpful talks. 

\section{Traces}
\subsection{Good polynomials}
Set
$$F_k = \begin{pmatrix} f_k &h_k \\ t_k &g_k \end{pmatrix}
= A^{x_1}B^{y_1}A^{x_2}B^{y_2}\cdots A^{x_k}B^{y_k}$$
where $A= \begin{pmatrix} 1 &2 \\ 0 &1 \end{pmatrix},\,B= \begin{pmatrix} 1 &0 \\ -2 &1 \end{pmatrix}.$
Then the trace $p_k = trF_k = f_k + g_k$ and all $f_k, h_k, t_k, g_k$ are the polynomials
in $2k$ variables $x_1, y_1, \cdots x_k, y_k$ with integer coefficients
whose explicit formula can be found in \cite{E}.

A sequence $\sigma$ of $2k$ signs $\pm$ can be viewed as a function
$\sigma : \{1, 2, \cdots 2k \}\rightarrow \{1, -1\}$.
For any polynomial $f$ in variables  $x_1, y_1, \cdots x_k, y_k$, set
$$f^{\sigma}=f(\si (1)(1+x_1), \si (2)(1+y_1), \cdots , \si (2k-1)(1+x_k), \si (2k)(1+y_{k}))$$

\begin{defi}{A polynomial f in $2k$ variables is said to be good if for arbitrary sequence $\sigma$ of $2k$ signs,
all the coefficients of $f^{\sigma}$ have the same sign. }\end{defi}

Let $Mat(2,2)$ be the set of $2 \times 2$ matrices over $\R$, the set of real numbers. Denote by $F_k^{\sigma}$ the matrix
$\begin{pmatrix} f_k^{\sigma} &h_k^{\sigma}\\ t_k^{\sigma} &g_k^{\sigma} \end{pmatrix}$. If $M= \begin{pmatrix} a &c \\ b &d \end{pmatrix} \in Mat(2,2)$,
then \begin{equation*} \begin{array}{rl}
tr(F_kM) &= af_k + bh_k + ct_k +dg_k\\
tr(F_k^{\sigma}M) &= af_k^{\sigma} + bh_k^{\sigma} + ct_k^{\sigma} +dg_k^{\sigma}
\end{array} \end{equation*}
Write

\begin{align*}
A_1= \begin{pmatrix} 1 &0 \\ 0 &0 \end{pmatrix}&&
A_2= \begin{pmatrix} 2 &1 \\ 0 &0 \end{pmatrix}&&
A_3= \begin{pmatrix} 2 &-1 \\ 0 &0 \end{pmatrix}, \\
A_4= \begin{pmatrix} 3 &2 \\ -2 &-1 \end{pmatrix}&&
A_5= \begin{pmatrix} 5 &2 \\ 2 &1 \end{pmatrix}&&
A_6= \begin{pmatrix} 5 &-2 \\ -2 &1 \end{pmatrix}.
\end{align*}
\ni Note that
\begin{equation}
A_4 + A_5 = 4A_2, \, A_4 + A_6= 4A_3^t, \, A_4^t + A_5= 4A_2^t, \, A_4^t + A_6 = 4A_3, A_2 + A_3 = 4A_1,\label{semi}
 \end{equation}
\begin{equation}
A_4= -A^{-1}B^{-1},\, A^t_4=-AB, \, A_5= AB^{-1},\, A_6=A^{-1}B. \label{eq}
\end{equation}
Let $S$ be a subset of $Mat(2,2)$, we have
\begin{prop} \label{p1} If $S$ satisfies that
\begin{description}
\item[p1)] $ a > 0$, for all $M= \begin{pmatrix} a &c \\ b &d \end{pmatrix} \in S$,
\item[P2)] $tr(CM)\geq 0$, for each $C \in \{A_4, A_{4}^t, A_5, A_6 \},\, M \in S$,where $D^t$ stands for the transpose of the matrix $D$,
\item[P3)]$CS\subseteq S, \text{for each} \,\, C \in \{A_4, A_{4}^t, A_5, A_6 \},$
\end{description}
then $af_k + bh_k + ct_k +dg_k$ is good, for every
$M= \begin{pmatrix} a &c \\ b &d \end{pmatrix}\in S, k\geq 1.$
\end{prop}

\begin{rema} \label{r1} $S$ satisfies the conditions P1), P2) P3) if and only if so does the cone
$Cone(S)\triangleq \{ \sum a_iM_i \,|\, a_i\geq 0, M_i \in S \}$. Furthermore, any set $S$ satisfying P1) possesses the property
that $af_k + bh_k + ct_k +dg_k$ is good, for every
$M= \begin{pmatrix} a &c \\ b &d \end{pmatrix}\in S, k\geq 1$ if and only if $Cone(S)$ satisfying P1) possesses the same property. {\rm The first assertion is obvious and the second one follows from the fact that the sign of the leading term of  $F_k^{\sigma}M) = af_k^{\sigma} + bh_k^{\sigma} + ct_k^{\sigma} +dg_k^{\sigma}$ with $a>0$ is independent of $a$ (see the proof of Lemma \ref{pass})}\end{rema}

We shall prove several lemmas before proving this proposition.
\subsection{Definition of  $M^{ij}$}
Let $\sigma$ be a sequence of $2k$ signs and let $\sigma_i, i=0, 1, 2, 3$, be the sequence of $2k+2$ signs
such that (a)\, $\sigma_i(j)= \sigma (j), \text{for each}\, 1\leq j\leq 2k$ and (b)\, $\sigma_0 (2k+1)=1, \sigma_0 (2k+2)=1,
\sigma_1 (2k+1)=1, \sigma_1 (2k+2)=-1, \sigma_2 (2k+1)=-1, \sigma_2 (2k+2)=1, \sigma_3 (2k+1)=-1, \sigma_3 (2k+2)=-1.$
Obviously every sequence of $2k+2$ signs equals to $\sigma_i$, for some $\sigma$ and $i$.
For any $M= \begin{pmatrix} a &c \\ b &d \end{pmatrix}$, set
\begin{equation*}\begin{array}{ll}
M^{00} &= 4A_1M, \;\; M^{01} = 2A_3M, \;\; M^{02} = 2A_2^tM, \;\; M^{03} = A_4^tM \\
M^{10} &= 4A_1M, \;\; M^{11} = 2A_2M, \;\; M^{12} = 2A_2^tM, \;\; M^{13} = A_5M \\
M^{20} &= 4A_1M, \;\; M^{21} = 2A_3M, \;\; M^{22} = 2A_3^tM, \;\; M^{23} = A_6M \\
M^{30} &= 4A_1M, \;\; M^{31} = 2A_2M, \;\; M^{32} = 2A_3^tM, \;\; M^{33} = A_4M.
\end{array}\end{equation*}

\begin{lemm}  For any $M= \begin{pmatrix} a &c \\ b &d \end{pmatrix} \in Mat(2,2)$ and $k\geq 1$,
\begin{equation}
tr(F_{k+1}^{\sigma_i}M) = (-1)^{\tau (i)}(x_{k+1}y_{k+1}tr(F_k^{\sigma}M^{i0}) + x_{k+1}tr(F_k^{\sigma}M^{i1}) + y_{k+1}tr(F_k^{\sigma}M^{i2}) + tr(F_k^{\sigma}M^{i3})
\label{sign} \end{equation}\end{lemm}
\ni where $i= 0, 1, 2, 3, \, \tau (0)=\tau (3)= 1, \, \tau (1)= \tau (2)= 0.$
\begin{proof}For $i=0$,
\begin{align*}
F_{k+1}^{\sigma_0} &= F_k^{\sigma}A^{1+x_{k+1}}B^{1+y_{k+1}}=F_k^{\sigma}\begin{pmatrix} 1 &2+2x_{k+1} \\ 0 &1 \end{pmatrix}
\begin{pmatrix} 1 &0 \\ -2-2y_{k+1} &1 \end{pmatrix}\\
&= F_k^{\sigma}\begin{pmatrix} -3-4x_{k+1}-4y_{k+1}-4x_{k+1}y_{k+1} &2+2x_{k+1} \\ -2-2y_{k+1} &1 \end{pmatrix}\\
&= -(x_{k+1}y_{k+1}F_k^{\sigma}(4A_1) + x_{k+1}F_k^{\sigma}(2A_3) + y_{k+1}F_k^{\sigma}(2A_2^t) + F_k^{\sigma}A_4^t)
\end{align*}
So, (\ref{sign}) holds for $i=0$. Similarly for $i=1, 2, 3$.
\end{proof}

\begin{lemm} \label{pass}Let $\sigma$ be a sequence of $2k$ signs and $M= \begin{pmatrix} a &c \\ b &d \end{pmatrix}$, with $a > 0$ and
assume further that the $(1,1)$ entry of $M^{ij}$ is also positive for every $i, j$.
Then for any $0\leq i \leq 3$, all the coefficients of $tr(F_{k+1}^{\sigma_i}M)$
are of the same sign if and only if all the coefficients of $tr(F_{k}^{\sigma}M^{ij})$
are of the same sign, for $j= 0, 1, 2, 3.$
\end{lemm}
\begin{proof}
Note that the explicit formula of $f_k, h_k, t_k, g_k$ in \cite{E} implies that
$deg(h_k)=deg(t_k)=2k-1, deg(g_k)=2k-2, deg(f_k)=2k$ and  the leading term of $f_k$ is $(-1)^k4^kx_1y_1\cdots x_ky_k$.
Hence if all the coefficients of $tr(F_{k}^{\sigma}M) = af_{k}^{\sigma}+ bh_{k}^{\sigma}+ ct_{k}^{\sigma} +dg_{k}^{\sigma}$, with $a > 0$,
are of the same sign, then all the coefficients have the  same sign with $(-1)^{k + \sharp(\sigma)}$ where $\sharp(\sigma)$
is the number of negative signs that $\sigma$ takes. Now the lemma follows immediately from $(\ref{sign})$.
\end{proof}

We have $f_1= 1- 4x_1y_1, \, h_1= 2x_1,\, t_1= -2y_1,\, g_1= 1.$
If set $F_0 = E$, the identity matrix,  then $(\ref{sign})$ also holds for $k=0$, i.e.
\begin{equation}
tr(F_1^{\sigma_i}M) = (-1)^{\tau (i)}(x_{1}y_{1}tr(M^{i0}) + x_{1}tr(M^{i1}) + y_{1}tr(M^{i2}) + tr(M^{i3})
 \label{sign1}\end{equation}
 where the sequences $\sigma_0, \sigma_1, \sigma_2, \sigma_3$ of $2$ signs are respectively $\{+, +\}, \{+, -\}, \{-, +\}, \{-, -\}$.
Let $M= \begin{pmatrix} a &c \\ b &d \end{pmatrix}$.

\begin{lemm}\label{leval1}
For $a>0$ and assume the $(1,1)$ entry of $M^{ij}$ is positive, then $af_1 + bh_1 + ct_1 + dg_1$ is good if and only if $tr(M^{i3}) \geq 0, i=0, 1, 2, 3.$\end{lemm}
\begin{proof}
It is easy to see by $(\ref{sign1})$ that $af_1 + bh_1 + ct_1 + dg_1$   is good if and only if  $tr(M^{ij}) \geq 0, \text{for all} \,i, j$. Now the lemma follows immediately from (\ref{semi}).

\end{proof}

\subsection{Proof of Proposition \ref{p1}}
 We prove it by induction on $k$. For $k=1$, $M= \begin{pmatrix} a &c \\ b &d \end{pmatrix} \in S$, we have $af_1 + bh_1 + ct_1 + dg_1$ is good by Lemma \ref{leval1}. Now assume  $af_k + bh_k + ct_k + dg_k$ is good,
for all $M= \begin{pmatrix} a &c \\ b &d \end{pmatrix} \in S$. One deduces first that  all $M^{ij}$
are contained $Cone(S)$ by (\ref{semi}) and the condition {p3)} that $S$ satisfies, and then that $tr(F_kM^{ij})$ is good
by the induction hypothesis and Remark \ref{r1}. Therefore $af_{k+1} + bh_{k+1} + ct_{k+1} + dg_{k+1}$ is good as well,  by Lemma \ref{pass}. \qed

\section{Proof of the Main Theorem}
\subsection{Decreasing matrices}
 A matrix $X = \begin{pmatrix} a &c\\b &d \end{pmatrix}$ is said to be decreasing if $|a| > |b| > |d| \,\text{and}\, |a| > |c| > |d|$, according to \cite{E}. The following lemma is proved in \cite{E}.
\begin{lemm}\label{dec}Let $X \in \Gamma (2)$ be decreasing and $m, n \in \Z\backslash \{0\}$, then $Y=A^mB^nX$ is decreasing.\end{lemm}

\subsection{Proof of Theorem \ref{main}}
Let $\Delta = \{ K_1K_2\cdots K_n \,| \, n\geq 0,  K_i = A_4, A_{4}^t, A_5, A_6 \}$ (for $n=0$, we mean the identity matrix $E$). By (\ref{eq}) every matrix except the identity matrix in $\Delta$ is decreasing by Lemma \ref{dec}. Now assume every word
of length $n, M=K_1K_2\cdots K_n =  \begin{pmatrix} a &c\\b &d \end{pmatrix}$ in $\Delta$, has the property that $a > 0$. Then for any word in $\Delta$ of length $n+1, M'=K_1K_2\cdots K_{n+1}= \begin{pmatrix} a' &c'\\b' &d' \end{pmatrix}$, it is easy to show that $a' > 0$ since $M$ is decreasing and $M'=MK_{n+1}$ with $K_{n+1}\in \{A_4, A_{4}^t, A_5, A_6 \}$. Hence we have proved, by induction, that for every
$C =  \begin{pmatrix} a &c\\b &d \end{pmatrix}\in \Delta,\, a > 0$, that is, $\Delta$ satisfies the condition P1).

In addition it is easy to see that the trace of a decreasing matrix whose $(1,1)$ entry is positive is always positive. Thus $\Delta$ satisfies the condition P2) as well. Meanwhile  $\Delta$ obviously satisfies the conditions  P3) by the definition of $\Delta$. Therefore  $af_k + bh_k + ct_k + dg_k$ is good, for all $k, \forall \, M= \begin{pmatrix} a &c \\ b &d \end{pmatrix} \in \Delta$, by  Proposition \ref{p1}. Now Theorem \ref{main} follows. \qed

\end{document}